\documentclass[12 pt]{amsart}
\usepackage{amsmath, amssymb, amsthm}
\usepackage{graphicx}
\title{Continuous Distributions Arising from the Three Gap Theorem}

\author{Gerem\'ias Polanco}
\address{School of Natural Science, Hampshire College, 893 West Street, Amherst, MA 01002, USA}
\email{gpeNS@hampshire.edu}

\author{Daniel Schultz}
\address{Penn State University Mathematics Dept.
University Park, State College, PA 16802, USA}
\email{dps23@math.psu.edu}

\author{Alexandru Zaharescu}
\address{Simion Stoilow Institute of Mathematics of the Romanian Academy, Research Unit 5, P.O.Box 1-764, 
RO-014700 Bucharest, Romania, and
Department of Mathematics, University of Illinois at Urbana-Champaign, 
1409 W. Green Street, Urbana, IL, 61801, USA
E-mail address: zaharesc@illinois.edu}
\email{zaharesc@illinois.edu}

\numberwithin{equation}{section}
\newtheorem{theorem}{Theorem}[section]

\newtheorem{lemma}[theorem]{Lemma}

\newtheorem{remark}[theorem]{Remark}
\begin{document}
\maketitle
\vspace{.5cm}
\begin{abstract}
The well known Three Gap Theorem states that there are at most three gap sizes in the sequence of fractional parts $\{\alpha n\}_{n<N}$. It is known that if one averages over $\alpha$, the distribution becomes continuous. We present an alternative approach, which establishes this averaged result and also provides good bounds for the error terms.
\end{abstract}
\section{Introduction}
Let $\alpha$ be a real number and consider the fractional parts $\{n
\alpha\}_{0\le n< Q}$ arranged in increasing order and placed on the
interval $[0,1]$ with $0$ and $1$ identified so that $Q$ gaps appear.
If we denote this sequence by $S_Q(\alpha)$ and consider the gaps between
consecutive elements, then it is well known that there are at most three gap sizes in
$S_Q(\alpha)$. For more on this topic the reader is referred to 
van Ravenstein \cite{dgapRavenstein}, S\'{o}s \cite{dgapSos} and \'Swierczkowski
\cite{dgapSwierczkowski} (see also \cite{dgapCobeli}, \cite{dgapGVZ} for some 
higher dimensional aspects of this phenomenon).
The sequence $\{n^d \alpha\}$, for $d\ge1$, has been
extensively studied. If $d>1$, the main reference is the work of Rudnick and
Sarnak \cite{dgapSarnak1}, where it is proved that for almost
all $\alpha$ the pair correlation of members of
the sequence is Poissonian (see also \cite{dgapZaharescu1}). 
Various other works have been done studying fractional parts of
polynomials. A non-exhaustive list of references includes Arhipov,
Karacuba, and \v{C}ubarikov\cite{dgapArhipov}, Baker and Harman
\cite{dgapBaker}, de Velasco \cite{dgapdevelasco}, Karacuba
\cite{dgapKaracuba}, Kovalevskaja \cite{dgapKovalevskaja},
Moshchevitin \cite{dgapmoshchevitin}, Schmidt \cite{dgapSchmidt}, and
Wooley \cite{dgapWooley}. The distribution of gaps and other aspects
of fractional parts has also been investigated for specific types of
numbers by Misevi\v{c}ius and Vakrinien\v{e} \cite{dgapMisevicius}, 
Zaimi \cite{dgapZaimi}, Dubickas \cite{dgapDubickas1}, \cite{dgapDubickas2},
Pillishshammer \cite{dgapPillichshammer}, and others.
Returning to the sequence $\{n \alpha\}$, it is known that if one averages over $\alpha$ 
the distribution becomes continuous. For the relevant literature the reader is referred to 
Bleher \cite{dgapBleher}, Mazel and Sinai \cite{dgapMazel}, and Greenman \cite{dgapGreenman}.
In particular, the density of the limiting distribution is computed in \cite{dgapGreenman}.

In the present paper we offer an alternative approach to proving such results, which
offers very satisfactory control over the error terms.
To proceed, we consider short averages of the nearest neighbor distribution of elements in
$S_Q(\alpha)$. We denote the $n^{\text{th}}$ element of
$S_Q(\alpha)$ by $\{\sigma_{n} \alpha\}$ and consider
the function:
\begin{equation}
\label{gk1def}
g_1^{\beta,\eta}(\lambda;Q)=\frac{1}{\eta}\int\limits_{\beta}^{\beta+\eta} {\frac{\#\{0 \le n <
Q: \{ \sigma_{n+1} \alpha \}-\{ \sigma_{n} \alpha \} \ge
\frac{\lambda}{Q}\} }{Q}} d\alpha\text{.}
\end{equation}
When $Q$ is
taken to be $1000$, $\alpha=1/3$, and $\eta=1/10$, the resulting graph is shown
in Figure 1.
\begin{figure}
\begin{center}
\includegraphics[width=5in]{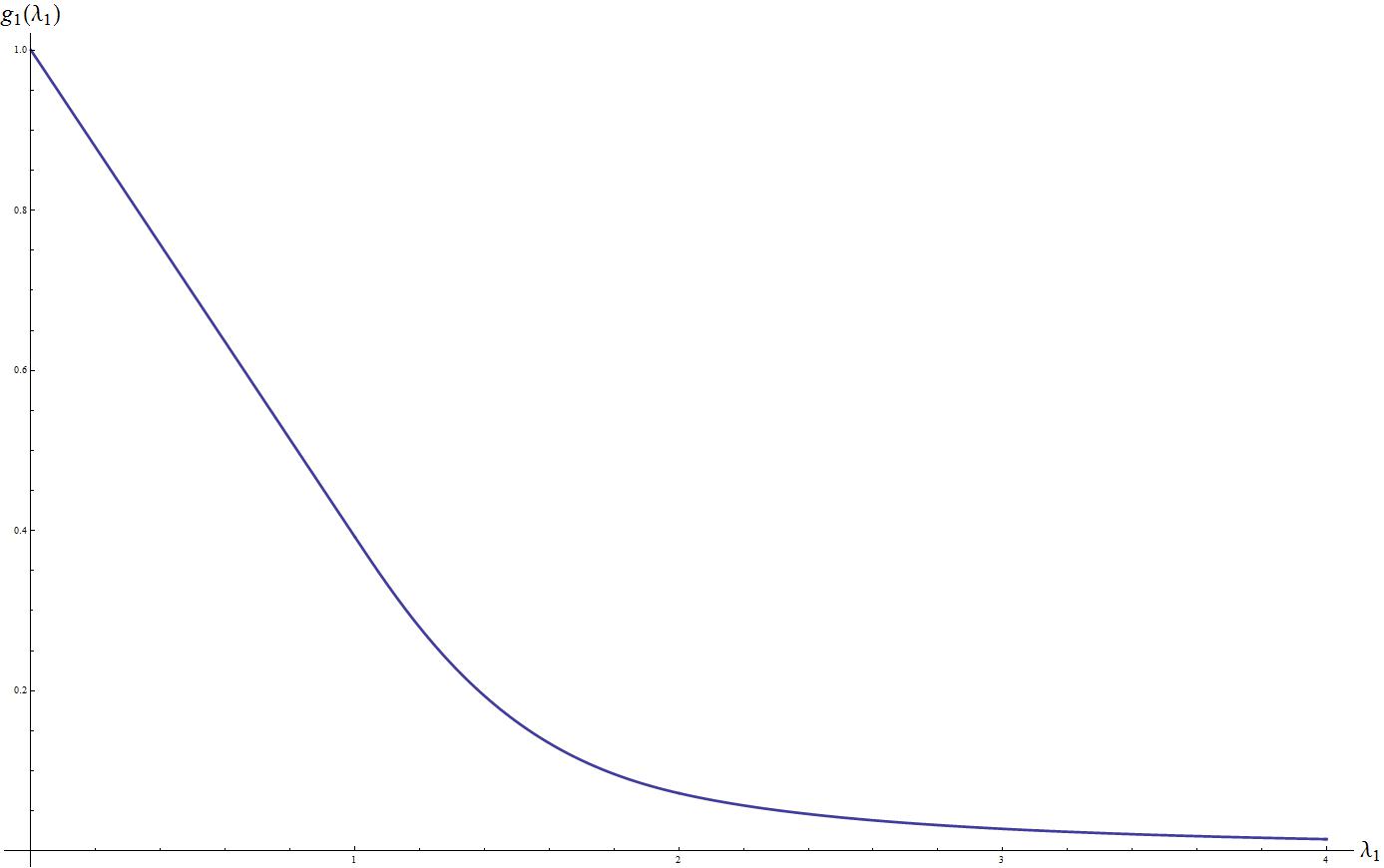}
\end{center}
\caption{$g_1^{\frac{1}{3},\frac{1}{10}}(\lambda;1000)$}
\label{figure1}
\end{figure}
Moreover, the graphs for different values of $\beta$ and $\eta$ appear to be identical.
Heuristically speaking, it is reasonable to
think that the subtleties lie on the nature of the sequence and how the gaps are related to its elements. 
It is natural to try and connect  the
elements of the sequence to the three possible gaps.
We establish such a connection using classical properties of
Farey fractions, that can be found in Hall \cite{dgapHall}, Hardy and
Wright \cite{dgapHardy}, and LeVeque \cite{dgapLeVeque}. We also use
other further developed properties connecting Farey fractions with
Kloosterman sums which have been applied to some useful asymptotics in
\cite{dgapZaharescu2}, Hall \cite{dgapHall2}, and in Hall and Tenenbaum
\cite{dgapHall3}. With these tools in hand we find an explicit formula
for the $g_1$ distribution function and show that this function is still approached as long as the size of the interval goes to zero no faster than $Q^{-1/2+\epsilon}$. More precisely we prove the following:\\

\begin{theorem}
\label{dgapthm1}
As $Q \to \infty$, the nearest neighbor distribution in \eqref{gk1def} is independent of $\beta$ and $\eta$, and we have
\begin{align*}
g_1(\lambda)&:=\lim_{Q \to \infty}{g_1^{\beta,\eta}(\lambda;Q)}\\
&=\frac{6}{\pi^2}\left\{
\begin{array}{ll}
\frac{\pi ^2}{6}-\lambda \text{,} & 0 < \lambda < 1 \\\\
\begin{array} {l}
\log ^2(2)-\frac{2 \pi ^2}{3}-1+\left(\frac{\lambda
}{2}-\frac{2}{\lambda }\right) \log \left(\frac{2-\lambda }{\lambda
-1}\right)+\\
\frac{3 \lambda}{2} \log \left(\frac{\lambda }{\lambda
-1}\right)
-\log \left(\frac{4}{\lambda }\right) \log (\lambda )+4
\operatorname{Li}_2\left(\frac{1}{\lambda }\right)+\\
2
\operatorname{Li}_2\left(\frac{\lambda }{2}\right)
\end{array}\text{,} & 1<\lambda <2 \\\\
-1+\left(\frac{\lambda}{2} -\frac{2}{\lambda }\right) \log
\left(\frac{\lambda -2}{\lambda -1}\right)+\frac{3 \lambda}{2} \log
\left(\frac{\lambda }{\lambda -1}\right)+\\
4
\operatorname{Li}_2\left(\frac{1}{\lambda }\right)-2
\operatorname{Li}_2\left(\frac{2}{\lambda }\right)\text{,} & 2<\lambda\text{,}
\end{array}
\right. 
\end{align*}
where the Dilogarithm is defined for $|z| \le 1$ by
\begin{equation*}
\operatorname{Li}_2(z)=\sum_{n=1}^{\infty}{\frac{z^n}{n^2}}\text{.}
\end{equation*}
Moreover, 
\begin{equation*}
\left| g_1^{\beta,\eta}(\lambda;Q)-g_1(\lambda)\right| \ll_{\epsilon} (1+\lambda) \frac{Q^{\epsilon-1/2}}{\eta} \text{.}
\end{equation*}
\end{theorem}

More generally, we consider the joint distribution
$g_k(\lambda_1,\lambda_2,\dots,\lambda_k)$, which describes the
average distribution of $k$-tuples of consecutive gaps over the
interval $[\beta,\beta+\eta]$, defined as
\begin{equation}
\label{dgap2}
g_k(\lambda_1,\lambda_2,\dots,\lambda_k)=\lim_{Q\to\infty}{g_k^{\beta,\eta}(\lambda_1,\dots,\lambda_k;Q)}\text{,}
\end{equation}
where
\begin{equation*}
g_k^{\beta,\eta}(\lambda_1,\dots,\lambda_k;Q)=\frac{1}{\eta}{\int\limits_{\beta}^{\beta+\eta}
{\frac{\#\{ \gamma_1 \cdots \gamma_k \in G_{Q,k}(\alpha) : \forall_{i
\le k} \ \gamma_i \ge \frac{\lambda_i}{Q} \} }{Q}} d\alpha}\text{,}
\end{equation*}
 and $G_{Q,k}(\alpha)$ denotes the list of $k$-sequences of consecutive
gaps in $S_Q(\alpha)$ (see the beginning of next section for an illustration of $G_{Q,K}(\alpha)$ for $K=1,2,3$).\\

When $Q$ is taken to be $1000$ the resulting graph is shown in Figure 2.
\begin{figure}
\begin{center}
\includegraphics[width=5in]{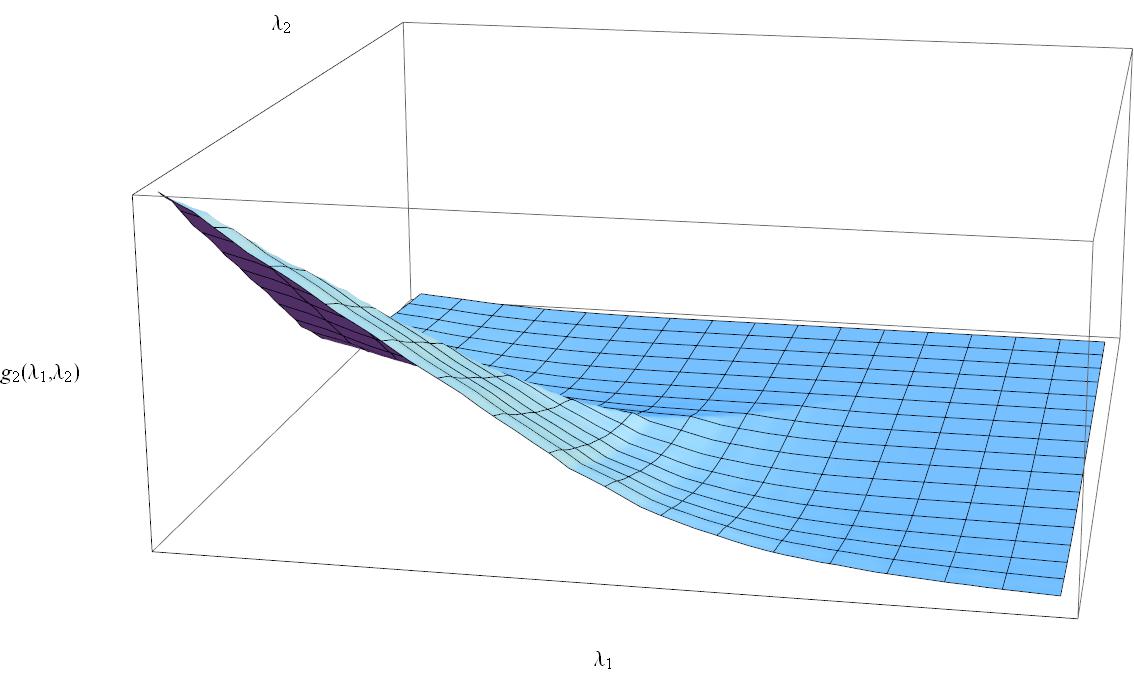}
\end{center}
\caption{$g_2(\lambda_1,\lambda_2)$}
\label{figure2}
\end{figure}
It can be proved that the functions $g_k(\lambda_1,\dots,\lambda_k)$
exist and are independent of the interval $[\beta,\beta+\eta]$ on which the average
is taken. Towards the end of this paper we concentrate in this
multidimensional case and find an explicit formula for
$g_2(\lambda_1,\lambda_2)$. We will deduce that the joint distribution
$g_2(\lambda_1,\lambda_2)$ is piecewise linear for
$\lambda_1+\lambda_2<1$, that is,
\begin{equation}
\label{g2example}
g_2(\lambda_1,\lambda_2)=1-\frac{6}{\pi^2}
\operatorname{Max}(\lambda_1,\lambda_2)-\frac{3}{\pi^2}
\operatorname{Min}(\lambda_1,\lambda_2) \text{,} \quad \text{for}
\quad \lambda_1+\lambda_2<1 \text{.}
\end{equation}
From a probabilistic point of view, our questions of interest deal with the set $S_Q({\alpha})$, whose average gap size is $Q^{-1}$, for large $Q$. In the one
dimensional case and a given $\lambda_1$, what is the probability that a randomly selected gap will be greater than $\lambda_1$ times the average gap? For the the $k$
dimensional case (say $k=2$) and a given $\lambda_1$ and $\lambda_2$, what is the probability that a gap and its
neighbor to the right are both greater than the average times
$\lambda_1$ and $\lambda_2$ respectively? Several papers have pointed
out this interpretation of the problems of our interest (see for example
\cite{dgapRudnick1}, \cite{dgapRudnick2}, \cite{dgapZaharescu3},
and the references there in). We
answer the two questions posed above by giving the cumulative
distribution function for the $1$-dimensional and 2-dimensional case.
One would expect that in the $k$-dimensional case the events are not
independent. This is made evident by the fact that
$g(\lambda_1,\lambda_2)$ is not the product of
$g(\lambda_1)$ and $g(\lambda_2)$ (by Theorem \ref{dgapthm1} and \eqref{g2example}).
The proof of these theorems have three key stages which will determine
the structure of the present paper. The next three sections of this
paper are devoted to the one dimensional case $g(\lambda_1)$. We first
establish a connection of the elements of the sequence of
fractional parts with Farey fractions in Section 2. The third section
of the paper is devoted to some lemmas that allow us to write a
sum over Farey fractions in a given region as an integral that is
easier to handle. The main ingredients of these lemmas are Kloosterman
sums. In the third section we write $g(\lambda_1)$ as a suitable sum
for which we can apply the lemmas from the third section. After some
other miscellaneous work we complete the proof of Theorem
\ref{dgapthm1}. The last two sections of this paper are devoted to the
joint distribution from $k$-dimensional spacing between consecutive
elements of the sequence of fractional parts, with special attention
to $k=2$.

\section{Key Connections to Farey Fractions}
In order to understand the limit given in definition \eqref{gk1def} we
must first understand the cardinality of the set in the integrand.
Specifically, given that we have three possible gaps sizes $A$, $B$ and $C$,
we need to have more information on the number of gaps of each size for any given $Q$, as well as a way to handle the contribution of each gap. This is accomplished in Lemma \ref{lemma1}, where we identify the two smaller gaps $A$ and $B$ with the closest Farey fractions on either side of $\alpha$. This will allow us to rewrite the integral from Theorem \ref{dgapthm1} in a form more useful for computations.\\

Let us illustrate the three gaps theorem in the case when $\alpha=\sqrt{2}$ and $Q=10$. The points $S_{10}(\sqrt{2})$ along
with the three gaps labeled $A$, $B$, and $C$ are shown in Figure 3.
\begin{figure}
\begin{center}
\includegraphics[width=5in]{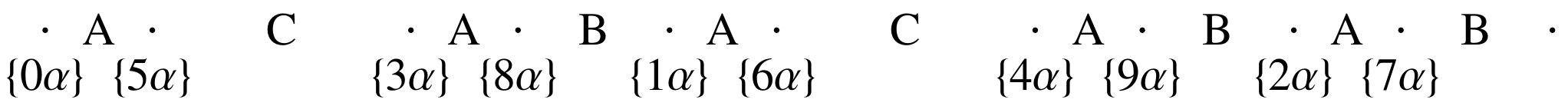}
\end{center}
\caption{$S_{10}(\sqrt{2})$}
\label{figure10}
\end{figure}
Let $G_{Q,k}(\alpha)$ denotes the list of $k$-sequences of consecutive
gaps in $S_Q(\alpha)$. When $k>1$ we allow the sequence to wrap around
$0$ so that exactly $Q$ tuples of gaps are considered. For example,
\begin{align*}
G_{10,1}(\sqrt{2})&=\{A,C,A,B,A,C,A,B,A,B\}\text{,}\\
G_{10,2}(\sqrt{2})&=\{AC,CA,AB,BA,AC,CA,AB,BA,AB,BA\}\text{,}\\
G_{10,3}(\sqrt{2})&=\{ACA,CAB,ABA,BAC,ACA,CAB,ABA,BAB,ABA,BAC\}\text{.}
\end{align*}
Also, in this case the permutation $\sigma$ satisfies
\begin{equation*}
\{\sigma_n\}_{n=0}^{9}=\{0,5,3,8,1,6,4,9,2,7\}\text{.}
\end{equation*}
\begin{lemma}
\label{lemma1}
For an irrational $\alpha$, the lengths of three gaps ($A$,$B$, and
$C$) generated by the sequence $\{n \alpha\}_{0\le n< Q}$ may be
computed as follows: Let us arrange the Farey fractions
$\{0\le\frac{a}{q}\le1|(a,q)=1 , \ q<Q \}$ in order on the interval
$[0,1]$ and choose consecutive fractions $\frac{a_1}{q_1}$ and
$\frac{a_2}{q_2}$ so that
\begin{equation*}
\frac{a_1}{q_1}<\alpha<\frac{a_2}{q_2} \text{.}
\end{equation*}
Then, the three gaps that can appear have lengths
\begin{align*}
A&=q_1 \alpha - a_1=\{q_1 \alpha\}\text{,}\\
B&=a_2 - q_2 \alpha = 1 - \{q_2 \alpha\} \text{,}\\
C&=A+B\text{,}\\
\end{align*}
and the function $\sigma$, which is a permutation of the set
$\{0,1,\dots,Q-1\}$ so that the sequence
\begin{equation*}
\{\sigma_0 \alpha\}, \quad \{\sigma_1 \alpha\}, \quad \dots, \quad
\{\sigma_{Q-1} \alpha\}
\end{equation*}
is in increasing order, satisfies:
\begin{align*}
\sigma_0&=0\text{,}\\
\sigma_{i+1}-\sigma_i&= \left\{\begin{array}{ll}
q_1\text{,} & \text{if } \sigma_i \in [0,Q-q_1) \quad\quad \text{ (A Gap)}\\
q_1-q_2 \text{,} & \text{if } \sigma_i \in [Q-q_1,q_2) \quad\quad \text{(C Gap)}\\
-q_2 \text{,} & \text{if } \sigma_i \in [q_2,Q) \quad\quad\quad\quad \text{(B Gap)}
\end{array}
\right. \text{.}
\end{align*}
\end{lemma}
\begin{remark}
\label{numbers}
This Lemma gives a simple formula for the number of gaps of each size. The number of $A$ gaps of each size is number of integers in the intervals on the right hand side of the recurrence relation for $\sigma_i$. Thus, the numbers of $A$ gaps, $B$ gaps and $C$ gaps are, respectively,
\begin{equation*}
Q-q_1, \quad Q-q_2, \quad q_1+q_2-Q \text{.}
\end{equation*}
This is in agreement with the fact that the total number of gaps is $Q$ and the total length of these gaps are $1$ since
\begin{gather*}
(Q-q_1) + (Q-q_2) + (q_1+q_2-Q) = Q\text{,}\\
(Q-q_1)A + (Q-q_2)B + (q_1+q_2-Q)C = 1\text{.}
\end{gather*}
This last equality is equivalent to the determinant property of Farey fractions:
\begin{equation*}
a_2 q_1 - a_1 q_2 =1\text{.}
\end{equation*}
Indeed, this property may be viewed as a corollary of Lemma \ref{lemma1}
\end{remark}
\begin{remark}
In \cite{dgapO'Bryant} O'Bryant developed a detailed account of
permutations that order fractional parts. Here we have used an
independent and simple idea to describe this permutations $\sigma_i$.
\end{remark}
\begin{proof}
Let us define a shifting operation that will associate each gap in the
sequence with the two gaps that have $0$ (or $1$) as an end point or
with the sum of these two gaps. Let the $A$ gap be the first gap, that
is, the one with $0$ as a left end point, and let the $B$ gap be the
last gap, that is, the gap with $1$ as a right end point. Suppose that
when the points $\{n \alpha\}$ for $0\le n < Q$ are arranged in
order, $\{n_1 x \}$ and $\{n_2 x \}$ appear consecutively so that the
interval $[\{n_1 x \},\{n_2 x \}]$ is a gap. If either $n_1$ or $n_2$
are $0$ then the gap is already the $A$ gap or the $B$ gap. If neither
$n_1$ nor $n_2$ is $0$, the gap $[\{n_1 x \},\{n_2 x \}]$ has the same
length as the interval $[\{(n_1-1) \alpha \},\{(n_2-1) \alpha \}]$ (if
$\{(n_1-1) \alpha \}>\{(n_2-1) \alpha \}$, we consider this interval
as wrapping around through $0$). This interval is also a gap except in
exactly one case: the case when the point $\{Q \alpha \}$ lies in the
gap $[\{n_1 x \},\{n_2 x \}]$. In this case the gap $[\{n_1 x \},\{n_2
x \}]$ is associated with the two smaller gaps $[\{(n_1-1) \alpha
\},\{(Q-1) \alpha\}]$ and $[\{(Q-1) \alpha\},\{(n_2-1)\alpha\}]$. This
proves that when an arbitrary starting gap is chosen, the shifting
process will divide the starting gap into two smaller gaps at most
once and that these smaller gaps will have the same size as the $A$
gap or the $B$ gap since this is where the shifting process
terminates.
We now know that the size of the $A$ gap is
\begin{equation*}
A=\underset{0< n<Q}{\operatorname{min}} \{n \alpha\}\text{,}
\end{equation*}
and the size of the $B$ gap is
\begin{equation*}
B=1-\underset{0< n<Q}{\operatorname{max}} \{n \alpha\}\text{.}
\end{equation*}
We need to show that the minimum and maximum are attained at $n=q_1$
and $n=q_2$ respectively, where $q_1$ and $q_2$ are as in the
statement of the lemma. Suppose there is another $q<Q$ with $\{q
\alpha\}<\{q_1 \alpha\}$ and write $\{q \alpha\}=q \alpha-a$. Then
\begin{equation}
\label{lemma1equ1}
0< q \alpha - a < q_1 \alpha - a_1\text{.}
\end{equation}
Since $\frac{a_1}{q_1}$ is the greatest Farey fraction less than
$\alpha$, we have
\begin{equation*}
\frac{a}{q}<\frac{a_1}{q_1}<\alpha\text{.}
\end{equation*}
In the case $q_1>q$, we have the trivial inequality
\begin{equation*}
\frac{a_1}{q_1}<\frac{a_1-a}{q_1-q}\text{,}
\end{equation*}
from which deduce that, since $\frac{a_1}{q_1}$ is the greatest Farey
fraction less than $\alpha$,
\begin{equation*}
\alpha<\frac{a_1-a}{q_1-q}\text{,}
\end{equation*}
which contradicts \eqref{lemma1equ1}. In the case $q_1<q$, we have the
similar inequality
\begin{equation*}
\frac{a-a_1}{q-q_1}<\frac{a_1}{q_1}\text{,}
\end{equation*}
from which we deduce that
\begin{equation*}
\frac{a-a_1}{q-q_1}<\alpha\text{,}
\end{equation*}
which also contradicts \eqref{lemma1equ1}. This proves that the
minimum is attained when $n=q_1$. The fact that the maximum is
attained at $n=q_2$ follows by symmetry since
\begin{equation*}
1-\underset{0< n<Q}{\operatorname{max}} \{n \alpha\}=\underset{0<
n<Q}{\operatorname{min}} \{n (1-\alpha)\}\text{.}
\end{equation*}
Now suppose that $\{\sigma_i \alpha \}$ is a point of $S_Q(\alpha)$.
If $\sigma_i \in [0,Q-q_1)$, then the $A$ gap, which is the interval
\begin{equation*}
[0,\{q_1 \alpha\}]\text{,}
\end{equation*}
can be shifted to the interval
\begin{equation*}
[\{\sigma_i \alpha\},\{(\sigma_i+q_1) \alpha\}]\text{,}
\end{equation*}
so that $\sigma_{i+1}=\sigma_{i}+q_1$ in this case. In the case
$\sigma_i \in [q_2,Q)$, the gap between $\sigma_i$ and $\sigma_{i+1}$
is a $B$ gap, so $\sigma_{i+1}=\sigma_{i}-q_2$ in this case. Finally,
when $\sigma_i \in [Q-q_1,q_2)$, the gap is a $C$ gap and so
$\sigma_{i+1}=\sigma_{i}+q_1-q_2$.
\end{proof}
\section{A Lemma Via Kloosterman Sums}
We will transform our integral into a sum of integrals over Farey arcs in a certain region. The resulting sum depends only on the denominators of the Farey fractions. This fact is crucial and explains the heuristic reason why the limit approaches the same distribution function in any interval. We present a general result for any short interval by borrowing Lemma 8 from \cite{BocaCoZa}. We have normalized the function and region by a factor of $Q$. 

\begin{lemma}
\label{lemma2}
\label{dgaplemma2}
Let
\begin{equation*}
\sum_{\Omega}{f}:=\sum_{\begin{array}{c}
(\frac{q_1}{Q},\frac{q_2}{Q})\in \Omega \\ \beta \le \frac{a_1}{q_1}
< \frac{a_2}{q_2} \le \beta+\eta \end{array} }{f} \text{.}
\end{equation*}
Then if $\Omega$ is a convex subregion of $[0,1]\times[0,1]$ with rectifiable boundary, and $f$ is a $C^1$ function on $\Omega$, we have
\begin{gather*}
\left| \frac{1}{\eta}\sum_{\Omega}{f
\left(\frac{q_1}{Q},\frac{q_2}{Q}\right) \frac{1}{Q^2}}-\frac{6}{\pi^2}\iint_{\Omega}f(x,y)dxdy \right|\\
\ll_{\epsilon}
\left(  \left\| \frac{\partial f}{\partial x} \right\|_{\infty}
+ \left\| \frac{\partial f}{\partial y} \right\|_{\infty}   \right)\frac{\operatorname{Area}(\Omega)\log Q}{\eta Q}+ \left\| f \right\|_{\infty} \frac{\left(1 +\operatorname{Len}(\partial\Omega) \right) \log Q}{\eta Q} + \\
\frac{m_{f} \left\| f \right\|_{\infty}}{\eta Q^{1/2-\epsilon}}
\text{,}
\end{gather*}
where $m_f$ is an upper bound for the number of intervals of monotonicity of each of the maps $y\rightarrow f(x,y)$.
\end{lemma}

\section{Proof of Theorem \ref{dgapthm1}}
We will handle the contribution of the three gaps of each type $A$,
$B$, and $C$ separately by partitioning the interval $[\beta,\beta+\eta]$ into Farey arcs with denominators strictly less than $Q$. Lemma
\ref{dgaplemma2} will then enable us to calculate the limit as $Q \to\infty$. We will deal with the error terms when we prove a general result for $g_k$ in the next section. In the interval $[\frac{a_1}{q_1},\frac{a_2}{q_2}]$, the $A$ gaps
contribute the amount
\begin{equation*}
\int\limits_{\frac{a_1}{q_1}}^{\frac{a_2}{q_2}}{\frac{1}{Q}\left\{\begin{array}{ll}
Q-q_1\text{,} & q_1 \alpha-a_1 \ge \lambda/Q \\ 0\text{,} & q_1
\alpha-a_1 < \lambda/Q \end{array} \right. d \alpha}\text{,}
\end{equation*}
since the $A$ gaps are $Q-q_1$ in number by Lemma \ref{lemma1}. If a
new integration variable $t$ defined by
\begin{equation}
\label{integral_substitution}
\alpha = \frac{a_1}{q_1}+\frac{t}{q_1 q_2}
\end{equation}
is introduced, the integral becomes
\begin{equation*}
\int\limits_{0}^{1}{\frac{dt}{Q q_1 q_2} \left\{\begin{array}{ll}
Q-q_1\text{,} & t \ge \frac{\lambda q_2}{Q} \\ 0\text{,} & t<
\frac{\lambda q_2}{Q} \end{array} \right.}=\left\{ \begin{array}{ll}
\frac{Q-q_1}{Q q_1 q_2}\left(1-\frac{\lambda q_2}{Q}\right)\text{,} &
0 \le \frac{\lambda q_2}{Q} \le 1 \\ 0\text{,} & 1 < \frac{\lambda
q_2}{Q} \end{array} \right. \text{.}
\end{equation*}
To compute the total contribution of the $A$ gaps to the function
$g_1(\lambda)$ we need to sum this expression over consecutive pairs
of Farey fractions in the interval $[\beta,\beta+\eta]$ as
\begin{equation*}
\frac{1}{\eta}\sum_{\beta \le \frac{a_1}{q_1} < \frac{a_2}{q_2} \le
\beta+\eta}{\left\{ \begin{array}{ll} \frac{Q-q_1}{Q q_1
q_2}\left(1-\frac{\lambda q_2}{Q}\right)\text{,} & 0 \le \frac{q_2}{Q}
\le \frac{1}{\lambda} \\ 0\text{,} & \frac{1}{\lambda} < \frac{q_2}{Q}
\end{array} \right.}\text{.}
\end{equation*}
By Lemma \ref{lemma2}, this sum converges to the integrals
\begin{equation*}
\frac{6}{\pi^2}\left\{\begin{array}{ll}\int\limits_{0}^{1}
\int\limits_{1-y}^{1}{\frac{(1-x)(1-\lambda y)}{x y}dx dy}\text{,} & 0
\le \lambda \le 1 \\ \int\limits_{0}^{\frac{1}{\lambda}}
\int\limits_{1-y}^{1}{\frac{(1-x)(1-\lambda y)}{x y}dx dy}\text{,} &
1<\lambda
\end{array}
\right. \text{,}
\end{equation*}
which further equals
\begin{equation}
\label{A_contribution}
\frac{6}{\pi^2}\left\{\begin{array}{ll}-1-\frac{\lambda }{2}+\frac{\pi
^2}{6}\text{,} & 0 \le \lambda \le 1 \\ -1-\frac{1}{2 \lambda
}+(1-\lambda ) \log
\left(1-\frac{1}{\lambda
}\right)+\operatorname{Li}_2\left(\frac{1}{\lambda }\right)\text{,} & 1<\lambda
\end{array}
\right. \text{,}
\end{equation}
and this is the contribution of the $A$ gaps to the function $g_1(\lambda)$.
The $B$ gaps contribute the amount
\begin{equation*}
\int\limits_{\frac{a_1}{q_1}}^{\frac{a_2}{q_2}}{\frac{1}{Q}\left\{\begin{array}{ll}
Q-q_2\text{,} & a_2-q_2 \alpha \ge \lambda/Q \\ 0\text{,} & a_2- q_2
\alpha < \lambda/Q \end{array} \right. d \alpha}
\end{equation*}
in the interval $[\frac{a_1}{q_1},\frac{a_2}{q_2}]$ since they are
$Q-q_2$ in number. The calculations to complete the total contribution
of the $B$ gaps are similar to those of the $A$ gaps, and
the total contribution is found to be the same as
\eqref{A_contribution}. By Lemma \ref{lemma1}, the $C$ gaps are
$q_1+q_2-Q$ in number, and thus contribute the amount
\begin{equation*}
\int\limits_{\frac{a_1}{q_1}}^{\frac{a_2}{q_2}}{\frac{1}{Q}\left\{\begin{array}{ll}
q_1+q_2-Q\text{,} & (q_1 \alpha-a_1)+(a_2-q_2 \alpha) \ge \lambda/Q \\
0\text{,} & (q_1 \alpha-a_1)+(a_2-q_2 \alpha) < \lambda/Q \end{array}
\right. d \alpha}
\end{equation*}
in the interval $[\frac{a_1}{q_1},\frac{a_2}{q_2}]$. With the
substitution \eqref{integral_substitution}, this integral becomes
\begin{equation*}
\int\limits_{0}^{1}{\frac{dt}{N q_1 q_2} \left\{\begin{array}{ll}
q_1+q_2-N\text{,} & \frac{1-t}{q_1}+\frac{t}{q_2} \ge
\frac{\lambda}{N} \\ 0\text{,} & \frac{1-t}{q_1}+\frac{t}{q_2} \ge
\frac{\lambda}{N} \end{array} \right.}
\end{equation*}
where the numerators $a_1$ and $a_2$ of the Farey fractions have
conveniently canceled out. If $q_2<q_1$, this last integral has the
evaluation
\begin{equation*}
\left\{\begin{array}{ll} \frac{q_1+q_2-N}{N q_1 q_2}\text{,} &
\frac{q_1}{N}<\frac{1}{\lambda} \\ \frac{q_1+q_2-N}{N q_1 q_2}
\frac{q_1 \left(N-\lambda q_2\right)}{N \left(q_1-q_2\right)} \text{,}
& \frac{q_2}{N}<\frac{1}{\lambda}<\frac{q_1}{N} \\ 0\text{,} &
\frac{1}{\lambda}<\frac{q_2}{N} \end{array} \right. \text{,}
\end{equation*}
while if $q_1<q_2$, the substitution $t \to 1-t$ shows that the
integral has the same evaluation with $q_1$ and $q_2$ switched:
\begin{equation*}
\left\{\begin{array}{ll} \frac{q_1+q_2-N}{N q_1 q_2}\text{,} &
\frac{q_2}{N}<\frac{1}{\lambda} \\ \frac{q_1+q_2-N}{N q_1 q_2}
\frac{q_2 \left(N-\lambda q_1\right)}{N \left(q_2-q_1\right)} \text{,}
& \frac{q_1}{N}<\frac{1}{\lambda}<\frac{q_2}{N} \\ 0\text{,} &
\frac{1}{\lambda}<\frac{q_1}{N} \end{array} \right. \text{.}
\end{equation*}
Thus, the total contribution of the $C$ gaps in the interval $[a,b]$ is
\begin{align*}
&\frac{1}{\eta}\sum_{\underset{q_2<q_1}{\beta \le
\frac{a_1}{q_1} < \frac{a_2}{q_2} \le \beta+\eta}}
{\left\{\begin{array}{ll} \frac{q_1+q_2-N}{N q_1 q_2}\text{,} &
\frac{q_1}{N}<\frac{1}{\lambda} \\ \frac{q_1+q_2-N}{N q_1 q_2}
\frac{q_1 \left(N-\lambda q_2\right)}{N \left(q_1-q_2\right)} \text{,}
& \frac{q_2}{N}<\frac{1}{\lambda}<\frac{q_1}{N} \\ 0\text{,} &
\frac{1}{\lambda}<\frac{q_2}{N} \end{array} \right.}\\
+&\frac{1}{\eta}\sum_{\underset{q_1<q_2}{\beta \le
\frac{a_1}{q_1} < \frac{a_2}{q_2} \le \beta+\eta}}
{\left\{\begin{array}{ll} \frac{q_1+q_2-N}{N q_1 q_2}\text{,} &
\frac{q_2}{N}<\frac{1}{\lambda} \\ \frac{q_1+q_2-N}{N q_1 q_2}
\frac{q_2 \left(N-\lambda q_1\right)}{N \left(q_2-q_1\right)} \text{,}
& \frac{q_1}{N}<\frac{1}{\lambda}<\frac{q_2}{N} \\ 0\text{,} &
\frac{1}{\lambda}<\frac{q_1}{N} \end{array} \right.}\text{.}
\end{align*}
By Lemma \ref{lemma2}, the first of these sums approaches the integrals
\begin{align*}
\frac{6}{\pi^2}\left\{\begin{array}{ll}\int\limits_{\frac{1}{2}}^{1}
\int\limits_{1-x}^{x}{\frac{x+y-1}{x y}dy dx}\text{,} & 0 \le \lambda
\le 1 \\ \int\limits_{\frac{1}{2}}^{\frac{1}{\lambda}}
\int\limits_{1-x}^{x}{\frac{x+y-1}{x y} dy
dx}+\int\limits_{\frac{1}{\lambda}}^{1}
\int\limits_{1-x}^{\frac{1}{\lambda}}{\frac{x+y-1}{y}\frac{1-\lambda
y}{x-y} dy dx}\text{,} & 1<\lambda < 2 \\
\int\limits_{0}^{\frac{1}{\lambda}}
\int\limits_{1-y}^{1}{{\frac{x+y-1}{y}\frac{1-\lambda y}{x-y}dx
dy}}\text{,} & 2 \le \lambda
\end{array}
\right. \text{,}
\end{align*}
and the second sum approaches the same set of integrals after the
change of variable $x \leftrightarrow y$ is made. Now, these integrals
can also be evaluated with the Dilogarithm, and the result is the
expression

\begin{equation}
\label{C_contribution}
\frac{6}{\pi^2}\left\{\begin{array}{ll}1-\frac{\pi ^2}{12}\text{,} & 0
\le \lambda \le 1 \\\\

\begin{array}{l} \frac{1}{2}-\frac{\pi
^2}{3}+\frac{\log ^2(2)}{2} + \frac{1}{2 \lambda }+\frac{ \lambda}{4}
\log \left(\frac{2}{\lambda }-1\right)-\frac{1}{\lambda }\log
\left(\frac{2-\lambda }{\lambda -1}\right)-\\
\log \left(\frac{\lambda -1}{\lambda}\right)
+\frac{1}{2} \log \left(\frac{\lambda }{4}\right) \log (\lambda
)+\operatorname{Li}_2\left(\frac{1}{\lambda
}\right)+\operatorname{Li}_2\left(\frac{\lambda }{2}\right)
\end{array}
\text{,} & 1<\lambda<2 \\\\

\frac{3}{4}-\frac{\pi ^2}{12}-\frac{\log ^2(2)}{2}+\frac{\log
(2)}{2}\text{,} & \lambda = 2\\\\

\begin{array}{l}\frac{1}{2}+\frac{1}{2 \lambda }+\frac{ \lambda}{4} \log
\left(1-\frac{2}{\lambda }\right)-\log
\left(1-\frac{1}{\lambda }\right)+\frac{1}{\lambda}\log
\left(\frac{\lambda -1}{\lambda
-2}\right)+\\
\operatorname{Li}_2\left(\frac{1}{\lambda
}\right)-\operatorname{Li}_2\left(\frac{2}{\lambda }\right)\end{array}\text{,} &
2<\lambda
\end{array}
\right.
\end{equation}

for the contribution of the first sum to the function $g_1(\lambda)$.
Thus, $g_1(\lambda)$ is the twice the sum of \eqref{C_contribution}
and \eqref{A_contribution}. Since
\begin{equation*}
\operatorname{Li}_{2}^{\prime}(z)=-\frac{\log (1-z)}{z}\text{,}
\end{equation*}
it is now straightforward to calculate $g_1^{\prime}(\lambda)$ and verify that it
is differentiable everywhere, that is, the following integral
representation is valid:
\begin{equation*}
g_1(\lambda)=\int\limits_{0}^{\frac{1}{\lambda}}{-g_1^{\prime}\left(\frac{1}{x}\right)\frac{d
x}{x^2}}\text{,}
\end{equation*}
where
\begin{equation*}
-g_1^{\prime}\left(\frac{1}{x}\right)=\frac{6}{\pi^2}\left\{\begin{array}{ll}
x+\log \frac{|1-x|^{2 (1-x)^2}}{|1-2 x|^{{\frac{1}{2}(1-2
x)^2}}}\text{,} & 0< x < 1\\
1 \text{,} & 1 < x
\end{array}
\right. \text{.}
\end{equation*}
\section{The Joint Distribution from $k$-Dimensional Spacing}
We now focus on generalizing Theorem \ref{dgapthm1}, but not by introducing a $k$-tuple of $\alpha$'s but
instead we consider a $k$-tuple of $\lambda$'s as in definition
\ref{dgap2}. In the previous theorem two key properties were used: the fact that
the number of gaps of each of the three sizes is a simple function of
$q_1$, and $q_2$, and the fact that all of the numerators of the Farey
fraction canceled out of the final calculations. The exact same
phenomenon happens in the general case for the function $g_k$. An
exact statement about the number of times a given sequence of
consecutive gap sizes appears is given in the following lemma. Set
\begin{equation*}
||I||=\operatorname{Length}(I \cap [0,1])\text{,}
\end{equation*}
and let $T$ be the triangle
\begin{equation*}
\{(x,y): x \le 1,\ y \le 1, \ x+y \ge 1 \} \text{.}
\end{equation*}
\begin{lemma}
\label{lemma10}
For any sequence $\gamma_1 \gamma_2 \cdots\gamma_k$ of the gap sizes
$A$, $B$, and $C$, there is a continuous function
\begin{equation*}
f_{\gamma_1\gamma_2 \cdots \gamma_k}(x,y) : T \to [0,1]
\end{equation*}
such that for $q_1$ and $q_2$ as in Lemma \ref{lemma1},
\begin{equation*}
{\frac{\#\{ \Gamma \in G_{Q,k}(\alpha) : \Gamma = \gamma_1 \gamma_2
\cdots\gamma_k \} }{Q}}=f_{\gamma_1 \gamma_2 \cdots \gamma_k}
\left(\frac{q_1}{Q},\frac{q_2}{Q}\right)\text{.}
\end{equation*}
That is, the average number of times the sequence $\gamma_1 \gamma_2
\cdots \gamma_k$ appears as a sequence of consecutive gap sizes in
$S_Q(\alpha)$ may be interpolated by a continuous function uniformly
in $Q$.
\end{lemma}
\begin{remark}
By Remark \ref{numbers}, we have
\begin{align*}
f_A(x,y)&=1-x\text{,}\\
f_B(x,y)&=1-y\text{,}\\
f_C(x,y)&=x+y-1\text{.}
\end{align*}
\end{remark}
\begin{proof}
By Lemma \ref{lemma1}, we may write $f_{\gamma_1 \gamma_2 \cdots
\gamma_k} \left(\frac{q_1}{Q},\frac{q_2}{Q}\right)$ as the length of
the intersection of $k$ intervals whose endpoints are continuous
functions of $q_1$ and $q_2$. The general truth of this fact will be
evident from the proof of the particular case $\gamma_1
\gamma_2 \cdots \gamma_k = ABC$, for example. By Lemma \ref{lemma1} we
have

\begin{align*}
&{\frac{\#\{ \Gamma \in G_{Q,3}(\alpha) : \Gamma = ABC\}}{Q}}\\
&=\frac{1}{Q} \# \left\{0 \le i <Q :
\begin{array}{l}
\{\sigma_{i+1} \alpha\} - \{\sigma_{i+0} \alpha\} = A\text{,}\\
\{\sigma_{i+2} \alpha\} - \{\sigma_{i+1} \alpha\} = B\text{,}\\
\{\sigma_{i+3} \alpha\} - \{\sigma_{i+2} \alpha\} = C
\end{array} \right\}\\
&=\frac{1}{Q} \# \left\{0 \le \sigma <Q : \begin{array}{l}
\sigma \in [0,Q-q_1)\text{,}\\
\sigma+q_1 \in [q_2,Q)\text{,}\\
\sigma+q_1-q_2 \in [Q-q_1,q_2)
\end{array} \right\}\\
&= \frac{1}{Q} || [0,Q-q_1) \cap [q_2-q_1,Q-q_1) \cap
[Q+q_2-2q_1,2q_1-q_1) || \text{.}
\end{align*}

In this case, we then have
\begin{equation*}
f_{ABC}(x,y)=||[0,1-x) \cap [y-x,1-x) \cap [1+y-2x,2x-y)||\text{,}
\end{equation*}
which is clearly a continuous function.
\end{proof}
\begin{theorem}
\label{gktheorem}
For any integer $k$, the limit
\begin{align*}
g_k(\lambda_1,\dots,\lambda_k)&=\lim_{Q \to \infty} {g_k^{\beta,\eta}(\lambda_1,\dots,\lambda_k;Q)}\\
&=\lim_{Q\to\infty}{\frac{1}{\eta}\int\limits_{\beta}^{\beta+\eta}
{\frac{\#\{ \gamma_1 \cdots \gamma_k \in G_{Q,k}(\alpha) : \forall_{i
\le k} \gamma_i \ge \frac{\lambda_i}{Q} \} }{Q}} d\alpha}
\end{align*}
exists and is independent of the interval $[\beta,\beta+\eta]$ on which the average
is taken. A formula for $g_k^{\beta,\eta}:=g_k^{\beta,\eta}(\lambda_1,\dots,\lambda_k;Q)$ and an error estimate is given by:

\begin{gather*}
\left| g_k^{\beta,\eta} -\frac{1}{\zeta(2)} \sum_{\gamma_1
\cdots \gamma_k \in \{A,B,C\}^k} \iint\limits_{T} f_{g_1 \cdots
g_k}(x,y) || \chi_{g_1}(\lambda_1) \cap \cdots \cap
\chi_{g_k}(\lambda_k) || \frac{d x d y}{x y}\right|\\
\ll_{k,\epsilon} (1+\lambda_1)\cdots(1+\lambda_k) \frac{Q^{\epsilon-1/2}}{\eta}
\end{gather*}

where
\begin{align*}
\chi_{A}(\lambda)=\chi_{A}(x,y,\lambda)&=[\lambda y,\infty)&\\
\chi_{B}(\lambda)=\chi_{B}(x,y,\lambda)&=(-\infty,1-\lambda x]&\\
\chi_{C}(\lambda)=\chi_{C}(x,y,\lambda)&=\left\{\begin{array}{ll} [ \frac{y (\lambda x-1)}{x-y},\infty ) & y<x
\\ (-\infty , 1- \frac{x (\lambda y-1)}{y-x} ] & x<y \end{array}
\right. \text{,}
\end{align*}
and $f_{\gamma_1,\dots,\gamma_k}(x,y)$ is defined in Lemma
\ref{lemma10}. The function $g_k(\lambda_1,\dots,\lambda_k)$ also
satisfies
\begin{equation*}
g_k(\lambda_1,\lambda_2,\dots,\lambda_{k-1},\lambda_k)=g_k(\lambda_k,\lambda_{k-1},\dots,\lambda_2,\lambda_1)\text{.}
\end{equation*}
\end{theorem}
\begin{proof}
We first divide the integral defining $g_k$ along the Farey fractions
in $[\beta,\beta+\eta]$ with denominator strictly less than $Q$:
\begin{equation*}
g_k^{\beta,\eta}=
\frac{1}{\eta}\sum_{ \beta \le \frac{a_1}{q_1} < \frac{a_2}{q_2} \le \beta+\eta }
\frac{1}{Q} \int\limits_{\frac{a_1}{q_1}}^{\frac{a_2}{q_2}} {\#\{
\gamma_1 \cdots \gamma_k \in G_{Q,k}(\alpha) : \forall_{i \le k}
\gamma_i \ge \frac{\lambda_i}{Q} \} } d \alpha
\end{equation*}
We next divide up the gap sequences among all of the $3^k$ possibilities:
\begin{gather*}
= \sum_{\gamma_1 \cdots \gamma_k \in \{A,B,C\}^k}
\frac{1}{\eta} \sum_{ \beta \le \frac{a_1}{q_1} < \frac{a_2}{q_2} \le \beta+\eta } \
\int\limits_{\frac{a_1}{q_1}}^{\frac{a_2}{q_2}} {\frac{\#\{ \Gamma \in
G_{Q,k}(\alpha) : \Gamma = \gamma_1 \cdots\gamma_k \} }{Q}} \\
\times \left\{ \begin{array}{ll} 1\text{,} & \forall_{i \le k} \
\gamma_i \ge \frac{\lambda_i}{Q} \\ 0\text{,} & \exists_{i \le k} \
\gamma_i < \frac{\lambda_i}{Q} \end{array} \right. d \alpha
\end{gather*}
The conditions that an $A$, $B$, or $C$ gap is bigger than
$\frac{\lambda}{Q}$ are by Lemma \ref{lemma1}:
\begin{align*}
q_1 \alpha - a_1 &\ge \frac{\lambda}{Q}\text{,} \quad \text{(for an A gap)}\\
a_2-q_2 \alpha &\ge \frac{\lambda}{Q}\text{,} \quad \text{(for a B gap)}\\
q_1 \alpha - a_1 + a_2-q_2 \alpha &\ge \frac{\lambda}{Q}\text{.} \quad
\text{(for a C gap)}
\end{align*}
Under the change of variable
\begin{equation*}
\alpha = \frac{a_1}{q_1}+\frac{t}{q_1 q_2}
\end{equation*}
these conditions become
\begin{align*}
t \ge \lambda \frac{q_2}{Q}&\text{,} \quad \text{(for an A gap)}\\
t \le 1-\lambda \frac{q_1}{Q}&\text{,} \quad \text{(for a B gap)}\\
\left(q_1-q_2\right) t\geq \frac{\lambda q_1 q_2}{Q}-q_2&\text{,}
\quad \text{(for a C gap)}
\end{align*}
respectively. Thus, we have
\begin{gather*}
\int\limits_{\frac{a_1}{q_1}}^{\frac{a_2}{q_2}}{\frac{\#\{ \Gamma \in
G_{Q,k}(\alpha) : \Gamma = \gamma_1 \cdots\gamma_k \} }{Q}} \left\{
\begin{array}{ll} 1\text{,} & \forall_{i \le k} \ \gamma_i \ge
\frac{\lambda_i}{Q} \\ 0\text{,} & \exists_{i \le k} \ \gamma_i <
\frac{\lambda_i}{Q} \end{array} \right. d \alpha\\
=f_{\gamma_1 \cdots \gamma_k}\left(\frac{q_1}{Q},\frac{q_2}{Q}\right)
||\chi_{\gamma_1}\left(\frac{q_1}{Q},\frac{q_2}{Q},\lambda_1\right)\cap
\cdots \cap \chi_{\gamma_k}\left(\frac{q_1}{Q},\frac{q_2}{Q},\lambda_k\right)||
\frac{1}{q_1 q_2}
\end{gather*}
So, finally,
\begin{equation*}
g_k^{\beta,\eta} (\lambda_1,\dots,\lambda_k;Q)= \frac{1}{\eta} \sum_{T}{f \left (\frac{q_1}{Q},\frac{q_2}{Q} \right)\frac{1}{Q^2}} \text{,}
\end{equation*}
where
\begin{equation*}
f(x,y)=f_{\gamma_1 \cdots
\gamma_k}(x,y) \frac{|| \chi_{\gamma_1}(x,y,\lambda_1) \cap \cdots \cap
\chi_{\gamma_k}(x,y,\lambda_k) ||}{x y} \text{.}
\end{equation*}
In order to apply Lemma \ref{dgaplemma2}, we need bounds for $f$ and its partial derivatives. Since $f$ is unbounded and its partials are unbounded on $T$, the domain must be shrunk to one of the form
\begin{equation*}
T_{\delta}=\{(x,y): x \le 1,\ y \le 1, \ x+y \ge 1+\delta, |x-y|>\delta \} \text{,}
\end{equation*}
and set $\left\| \cdot \right\|_{\infty}$ to be the supremum norm on $T_{\delta}$. Now, the functions $f_{\gamma_1 \cdots \gamma_k}$ are relatively tame; they satisfy
\begin{gather*}
\left\| f_{\gamma_1 \cdots \gamma_k} \right\|_{\infty} \le  1\\
\left\| \frac{\partial f_{\gamma_1 \cdots \gamma_k}}{\partial x} \right\|_{\infty}+ \left\| \frac{\partial f_{\gamma_1 \cdots \gamma_k}}{\partial y} \right\|_{\infty} \le 2k
\end{gather*}
Hence we focus on the functions
\begin{equation*}
F_{\gamma_1 \cdots \gamma_k}(x,y)=\frac{|| \chi_{\gamma_1}(x,y,\lambda_1) \cap \cdots \cap
\chi_{\gamma_k}(x,y,\lambda_k) ||}{x y} \text{.}
\end{equation*}
Since $0 \le F_{\gamma_1 \cdots \gamma_k}(x,y) \le \frac{1}{x y}$,
\begin{equation*}
\left\| F_{\gamma_1 \cdots \gamma_k} \right\|_{\infty} \le \frac{1}{\delta}\text{.}
\end{equation*}
Next, the fact that
\begin{equation*}
\left\| \frac{\partial F_{\gamma_1 \cdots \gamma_k}}{\partial x} \right\|_{\infty}+ \left\| \frac{\partial F_{\gamma_1 \cdots \gamma_k}}{\partial y} \right\|_{\infty} \ll_k \frac{\Lambda}{\delta^2} \text{,}
\end{equation*}
where $\Lambda=(1+\lambda_1)\cdots (1+\lambda_k)$, follows from the following table:
\begin{center}
\begin{tabular}{lll}
   \hline
    $h(x,y)$  & bound for $\left\| h \right\|_{\infty}$ & bound for $\left\| \frac{\partial h}{\partial x} \right\|_{\infty}+ \left\| \frac{\partial h}{\partial y} \right\|_{\infty}$ \\
    \hline
    \hline & &  \\ [-1.2ex]
    $|| \chi_{A}(x,y,\lambda)||$  & $1$ & $\lambda$                \\[1.0ex]
    $|| \chi_{B}(x,y,\lambda)||$  & $1$ & $\lambda$                 \\[1.0ex]
    $|| \chi_{C}(x,y,\lambda)||$  & $1$ & $\left\{\begin{array}{ll} 0 \text{,} & \lambda < 1 \\ 2\lambda \left(1+\frac{1}{\lambda \delta}\right)\text{,} & 1<\lambda<2 \\ 2\lambda \left(1+\frac{1}{\lambda+\lambda \delta-2}\right)\text{,} & 2<\lambda
\end{array}
\right.$\\[1.0ex]
    $\frac{1}{x y}$  & $\frac{1}{\delta}$ & $\frac{2}{\delta^2}$\\[1.0ex]
   \hline
\end{tabular}
\end{center}
In summary, we have
\begin{gather*}
\left\| f \right\|_{\infty} \ll_k  \frac{1}{\delta}\text{,}\\
\left\| \frac{\partial f}{\partial x} \right\|_{\infty}+ \left\| \frac{\partial f}{\partial y} \right\|_{\infty} \ll_k \frac{\Lambda}{\delta^2}\text{.}
\end{gather*}
Therefore, by Lemma \ref{dgaplemma2} with $m_f=O_k(1)$,
\begin{gather*}
\left|\frac{1}{\eta}\sum_{T_{\delta}}{f\left(\frac{q_1}{Q},\frac{q_2}{Q}\right)\frac{1}{Q^2}}-\frac{1}{\zeta(2)}  \iint\limits_{T_{\delta}} {f(x,y) d x d y} \right|\\
\ll_{k,\epsilon}  \frac{\Lambda \log Q}{\eta \delta^2 Q}+\frac{\log Q}{\eta \delta Q}+\frac{Q^{\epsilon}}{\eta \delta Q^{1/2}}\\
\ll_{k,\epsilon}\frac{Q^{\epsilon-1}}{\eta} \left( \frac{\Lambda}{\delta^2}+Q^{1/2}\right)\text{.}
\end{gather*}
Finally,
\begin{align*}
&\left|\frac{1}{\eta}\sum_{T}{f\left(\frac{q_1}{Q},\frac{q_2}{Q}\right)\frac{1}{Q^2}}-\frac{1}{\zeta(2)}  \iint\limits_{T} {f(x,y) d x d y} \right|\\
&\le \left|\frac{1}{\eta}\sum_{T_{\delta}}{f\left(\frac{q_1}{Q},\frac{q_2}{Q}\right)\frac{1}{Q^2}}-\frac{1}{\zeta(2)}  \iint\limits_{T_{\delta}} {f(x,y) d x d y} \right|\\
&+\frac{1}{\eta}\left|\sum_{T \setminus T_{\delta}}{f\left(\frac{q_1}{Q},\frac{q_2}{Q}\right)\frac{1}{Q^2}}\right|+  \left| \frac{1}{\zeta(2)}  \iint\limits_{T \setminus T_{\delta}} {f(x,y) d x d y} \right|\\
&\ll_{k,\epsilon}\frac{Q^{\epsilon-1}}{\eta} \left( \frac{\Lambda}{\delta^2}+Q^{1/2}\right) + \frac{\delta}{\eta} \text{.}
\end{align*}
This error term is optimized with the choice $\delta=Q^{-\epsilon}$, leading to
\begin{equation*}
\left|\frac{1}{\eta}\sum_{T}{f\left(\frac{q_1}{Q},\frac{q_2}{Q}\right)\frac{1}{Q^2}}-\frac{1}{\zeta(2)}  \iint\limits_{T} {f(x,y) d x d y} \right|\ll_{k,\epsilon}\frac{\Lambda Q^{\epsilon-1/2}}{\eta} \text{.}
\end{equation*}
The symmetry of the function $g_k$ can be obtained from the symmetry of the function $g_k^{0,1}(\lambda_1,\dots,\lambda_k;Q)$ as:
\begin{align*}
g_k(\lambda_1,\dots,\lambda_k;Q)&=\lim_{Q \to \infty}{g_k^{0,1}(\lambda_1,\dots,\lambda_k;Q)}\\
&=\lim_{Q \to \infty}{g_k^{0,1}(\lambda_k,\dots,\lambda_1;Q)}\\
&=g_k(\lambda_k,\dots,\lambda_1;Q) \text{.}
\end{align*}
The function $g_k^{0,1}(\lambda_1,\dots,\lambda_k;Q)$ has this symmetry property because the sequence of gaps in $S_Q(1-\alpha)$ is the reverse of the sequence of gaps in $S_Q(\alpha)$ with the $A$ and $B$ gaps switched.
\end{proof}

\section{Explicit Formula for $g_2(\lambda_1,\lambda_2)$}
Theorem \ref{gktheorem} gives a formula for
$g_2(\lambda_1,\lambda_2)$ as an integral over the region $T$. By
dividing up $[0,\infty) \times [0,\infty)$ into the $14$ regions shown in Figure 4,
\begin{figure}
\begin{center}
\includegraphics[width=5in]{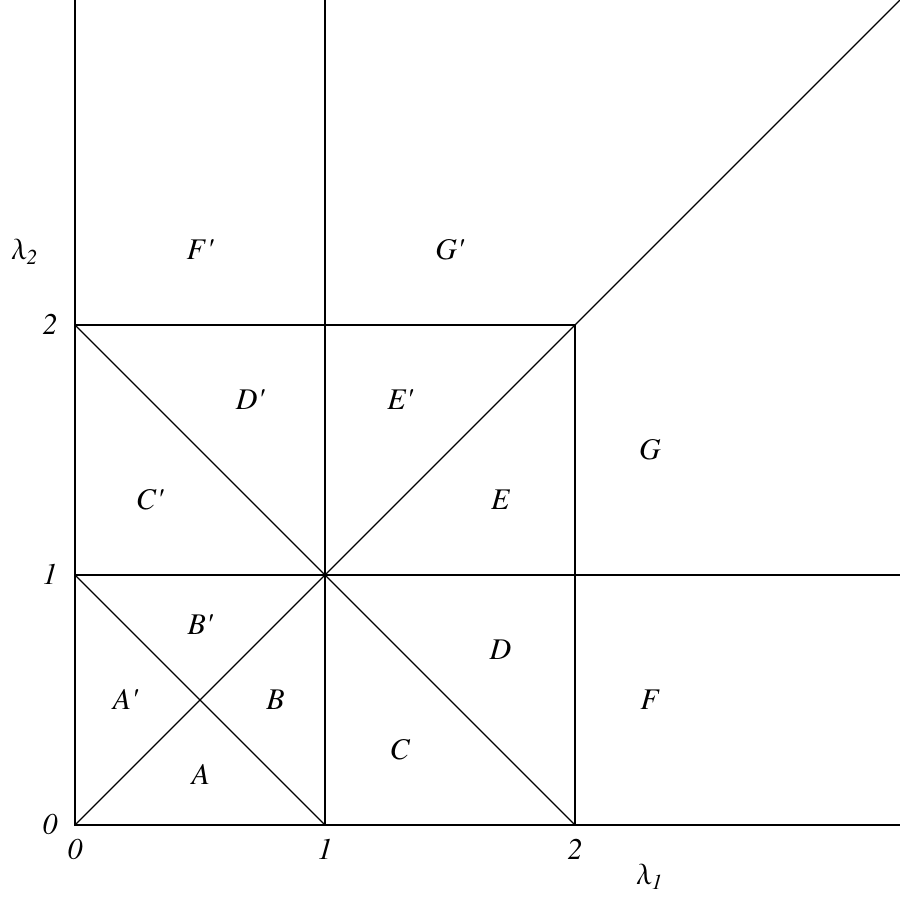}
\end{center}
\label{figure4}
\caption{regions on which $g_2(\lambda_1,\lambda_2)$ is $C^{\infty}$ smooth}
\end{figure}
we calculate $g_2(\lambda_1,\lambda_2)$ on each of the regions $A$,
$B$, $C$, $D$, $E$, $F$, and $G$. The resulting expressions are given in the following theorem.

\begin{theorem}
The explicit formula for $g_2(\lambda_1,\lambda_2)$ on each of the regions $A$,
$B$, $C$, $D$, $E$, $F$, and $G$ in Figure 4 are as follows. The value of
$g_2(\lambda_1,\lambda_2)$ on $A'$, $B'$, $C'$, $D'$, $E'$, $F'$, and
$G'$ may be found using the symmetry property
\begin{equation*}
g_2(\lambda_1,\lambda_2)=g_2(\lambda_2,\lambda_1)
\end{equation*}
given in Theorem \ref{gktheorem}.

\begin{align*}
\frac{\pi^2}{6}g_2(\lambda_1,\lambda_2)|_A &=\frac{\pi ^2}{6}-\lambda
_1-\frac{\lambda _2}{2}\text{,}\\
\frac{\pi^2}{6}g_2(\lambda_1,\lambda_2)|_B &= \frac{\pi
^2}{6}-2+\lambda _1+\frac{3 \lambda _2}{2}-2
\operatorname{Li}_2\left(\lambda _1\right)+2
\operatorname{Li}_2\left(1-\lambda _2\right)\\&+\left(\frac{1}{\lambda _2}+\frac{1}{\lambda
_1}\right) \log \left(\lambda _1+\lambda _2\right)+\left(\frac{1}{\lambda _1}-\lambda _1\right) \log
\left(\frac{1-\lambda _1}{\lambda _2}\right)\\&+\log
\left(\frac{1-\lambda _2}{\lambda
_1}\right) \left(-\lambda _2+\frac{1}{\lambda _2}+2 \log \left(\lambda
_2\right)\right)
\end{align*}

\begin{align*}
\frac{\pi^2}{6}g_2(\lambda_1,\lambda_2)|_C &=\frac{3 \lambda
_2}{2}-\frac{\pi ^2}{3}-1+\log ^2(2)+\left(\frac{2}{\lambda _1}-2 \lambda _1\right) \log \left(\lambda
_1-1\right)\\&+\left(\frac{\lambda _1}{2}-\frac{2}{\lambda
_1}\right) \log \left(2-\lambda _1\right)+\left(\frac{\lambda
_1}{2}-\frac{2}{\lambda _2}\right) \log \left(\lambda _1\right)\\
&+\left(\frac{1}{\lambda _2}+\frac{1}{\lambda _1}\right)
\log \left(\lambda _1+\lambda _2\right)\\
&+\left(\lambda _1-\lambda
_2+\frac{1}{\lambda _2}-\frac{1}{\lambda _1}\right) \log \left(\lambda
_1-\lambda _2\right)\\
&+\log ^2\left(\lambda
_1\right)+\left(\lambda _2-2 \log \left(2\lambda _2\right)\right) \log \left(\lambda _1\right)\\&+\left(2 \log \left(\lambda
_1-\lambda _2\right)-2 \log \left(1-\lambda
_2\right)\right) \log \left(\lambda _2\right)\\
&+2 \operatorname{Li}_2\left(\frac{1}{\lambda _1}\right)+2
\operatorname{Li}_2\left(\frac{\lambda _1}{2}\right)-2
\operatorname{Li}_2\left(\lambda _2\right)\\
&+2\operatorname{Li}_2\left(\frac{\lambda _2-1}{\lambda _2-\lambda
_1}\right)-2 \operatorname{Li}_2\left(\frac{\lambda _1 \left(\lambda
_2-1\right)}{\lambda _2-\lambda _1}\right)\text{,}\\
\end{align*}

\begin{align*}
\frac{\pi^2}{6}g_2(\lambda_1,\lambda_2)&|_D =-\frac{5 \pi ^2}{6}
-\frac{\lambda _1}{2}+\lambda _2+\left(-\lambda _1+\lambda
_2+\frac{2}{\lambda _1}\right) \log (2)\\
&+\left(\frac{3}{\lambda _1}-3 \lambda _1\right) \log
\left(\lambda _1-1\right)+\left(\frac{2}{\lambda _1}-\frac{\lambda
_1}{2} - 2 \log (2) \right) \log \left(\lambda _2\right)\\
&+\left(\lambda
_1-\frac{4}{\lambda _1}\right) \log \left(2-\lambda
_1\right)+2 \left(\lambda _1-\lambda _2+\frac{1}{\lambda _2}-\frac{1}{\lambda
_1}\right) \log \left(\lambda _1-\lambda _2\right)\\
&+\left(\frac{\lambda
_1}{2}+\lambda
_2-\frac{2}{\lambda _2}\right) \log \left(\lambda
_1\right)+\left(\lambda _2-\frac{1}{\lambda _2}\right) \log
\left(1-\lambda _2\right)\\
&+\log ^2\left(\lambda _1\right)-2 \log
\left(-\lambda _1+\lambda _2+2\right) \log \left(\lambda _1\right)\\
&+2\log \left(\lambda _1-\lambda _2\right) \log \left(\lambda _2\right)+2
\log (2) \log \left(2-\lambda_1+\lambda _2\right)\\
&+4 \operatorname{Li}_2\left(\frac{1}{\lambda _1}\right)+2
\operatorname{Li}_2\left(\frac{\lambda _1}{2}\right)+2
\operatorname{Li}_2\left(1-\lambda _2\right)+2
\operatorname{Li}_2\left(\lambda _2\right)\\
&-2\operatorname{Li}_2\left(\frac{2 \lambda _2}{\lambda _1
\left(2-\lambda _1+\lambda _2\right)}\right)+2
\operatorname{Li}_2\left(\frac{\lambda _2-1}{\lambda _2-\lambda
_1}\right)\\
&+2 \operatorname{Li}_2\left(\frac{\lambda _2}{2-\lambda
_1+\lambda _2}\right)-2 \operatorname{Li}_2\left(\frac{2-\lambda _1+\lambda _2}{2}\right)\\
&-2\operatorname{Li}_2\left(\frac{\lambda _1-\lambda _1 \lambda
_2}{\lambda _1-\lambda _2}\right)\text{,}
\end{align*}

\begin{align*}
\frac{\pi^2}{6}g_2(\lambda_1,\lambda_2)&|_E =\frac{1}{6} \left(-3
\lambda _1-2 \pi ^2+6\right)+\left(-\lambda _1+\lambda
_2+\frac{2}{\lambda _1}\right) \log (2)-\log ^2(2)\\
&+\log ^2\left(\lambda _1\right)+\left(\frac{\lambda _1}{2}-\lambda
_2+\frac{1}{\lambda _2}-\frac{2}{\lambda _1}+2 \log (2)\right) \log \left(2
\lambda _2-\lambda _1\right)\\
&+\left(\frac{1}{\lambda _1}-\lambda
_1\right) \log \left(\lambda _1-1\right)+\left(\frac{\lambda _1}{2}-\frac{2}{\lambda
_1}\right) \log \left(2-\lambda _1\right)\\
&+\left(\frac{\lambda
_1}{2}-\frac{1}{\lambda _2}\right) \log \left(\lambda
_1\right)+\left(\lambda _2-\frac{1}{\lambda _2}\right) \log
\left(\lambda _2-1\right)\\
&+2\log (2) \left(\log \left(2-\lambda _1+\lambda _2\right)-\log
\left(-\lambda _1+\lambda _2+1\right)\right)\\
&+\log \left(\lambda
_2\right) \left(-\frac{\lambda _1}{2}+\frac{2}{\lambda _1}-2 \log
(2)\right)\\
&+\left(2 \log
\left(\lambda _2\right)+2 \log \left(-\lambda _1+\lambda _2+1\right)\right)\log\left(\lambda _1\right)\\
&-\left(2\log \left(-\lambda _1+\lambda _2+2\right)+2 \log \left(2 \lambda_2-\lambda _1\right)\right) \log\left(\lambda _1\right)\\
&+2 \operatorname{Li}_2\left(\frac{1}{\lambda _1}\right)+2
\operatorname{Li}_2\left(\frac{\lambda _1}{2}\right)-2
\operatorname{Li}_2\left(\frac{\lambda _1-2 \lambda _2}{2
\left(\lambda _1-\lambda _2-1\right)}\right)\\
&+2
\operatorname{Li}_2\left(\frac{\lambda _1-2 \lambda _2}{\lambda _1
\left(\lambda _1-\lambda _2-1\right)}\right)-2\operatorname{Li}_2\left(\frac{\lambda _1-\lambda _2-1}{\lambda _1-2
\lambda _2}\right)\\
&-2 \operatorname{Li}_2\left(-\frac{2 \lambda
_2}{\lambda _1 \left(\lambda _1-\lambda _2-2\right)}\right)+2
\operatorname{Li}_2\left(\frac{\left(\lambda _1-\lambda _2-1\right)
\lambda _2}{\lambda _1-2 \lambda _2}\right)\\
&+2
\operatorname{Li}_2\left(\frac{\lambda _2}{-\lambda _1+\lambda
_2+2}\right)-2
\operatorname{Li}_2\left(\frac{1}{2} \left(-\lambda _1+\lambda
_2+2\right)\right)\text{,}
\end{align*}

\begin{align*}
\frac{\pi^2}{6}g_2(\lambda_1,\lambda_2)&|_F = \lambda _2-\frac{\pi
^2}{2}-1+\log ^2(2)+2 \log (2) \log \left(\lambda _1-\lambda
_2-1\right)\\
&+\left(\frac{2}{\lambda _1}-2 \lambda _1\right) \log
\left(\lambda
_1-1\right)+\left(\frac{\lambda _1}{2}-\frac{2}{\lambda _1}\right)
\log \left(\lambda _1-2\right)\\
&+2 \left(\lambda _1-\lambda _2+\frac{1}{\lambda _2}-\frac{1}{\lambda
_1}\right) \log
\left(\lambda _1-\lambda _2\right)+\left(\lambda _2-\frac{1}{\lambda
_2}\right) \log \left(\lambda _1\right)\\
&+\log \left(\lambda _1-2 \lambda _2\right) \left(-\frac{\lambda
_1}{2}+\lambda _2-\frac{1}{\lambda _2}+\frac{2}{\lambda _1}-2 \log (2)\right)\\
&+2 \log \left(\lambda _1\right) \left(\log \left(\lambda _1-\lambda
_2\right)-\log \left(\lambda
_1-\lambda _2-1\right)-\log \left(\lambda _2\right)\right)\\
&+\log
\left(\lambda _1-\lambda _2\right) \log \left(\lambda _2\right)+2 \operatorname{Li}_2\left(\frac{1}{\lambda _1}\right)+2
\operatorname{Li}_2\left(1-\lambda _2\right)\\
&+2
\operatorname{Li}_2\left(\frac{\lambda _1-2 \lambda _2}{2
\left(\lambda _1-\lambda _2-1\right)}\right)-2
\operatorname{Li}_2\left(\frac{\lambda _1-2 \lambda
_2}{\lambda _1 \left(\lambda _1-\lambda _2-1\right)}\right)\\
&+2 \operatorname{Li}_2\left(\frac{\lambda _1-\lambda _2-1}{\lambda
_1-2 \lambda _2}\right)+2 \operatorname{Li}_2\left(\lambda _2\right)-2
\operatorname{Li}_2\left(\frac{\left(\lambda _1-\lambda _2-1\right)
\lambda _2}{\lambda _1-2 \lambda _2}\right)\\
&+2 \operatorname{Li}_2\left(\frac{\lambda _2-1}{\lambda _2-\lambda _1}\right)-2
\operatorname{Li}_2\left(\frac{\lambda _1-\lambda _1 \lambda
_2}{\lambda _1-\lambda _2}\right)\text{,}\\
\frac{\pi^2}{6}g_2(\lambda_1,\lambda_2)&|_{G} = 0 \text{.}
\end{align*}
\end{theorem}
\bibliographystyle{model1a-num-names}

\end{document}